\documentclass[final,3p,times]{article}

\usepackage[a4paper,left=25mm,top=25mm,right=20mm,bottom=30mm]{geometry}

\usepackage{authblk}
\usepackage[numbers, sort&compress]{natbib}

\usepackage[colorlinks]{hyperref}
\hypersetup{
	citecolor=blue,
   linkcolor=red,
}

\usepackage{amsmath, amssymb, amsfonts, amsfonts, amsthm,latexsym}
\usepackage[noend]{algorithmic}
\usepackage{algorithm}

\usepackage{graphics}
\usepackage{enumerate}
\usepackage[usenames]{color}
\usepackage{mathtools}
\usepackage[normalem]{ulem}
\newtheorem{theorem}{Theorem}[section]
\newtheorem{lemma}[theorem]{Lemma}
\newtheorem{definition}{Definition}[section]
\newtheorem{corollary}[theorem]{Corollary}

\usepackage{url}
\RequirePackage{hyperref}

\DeclareSymbolFont{yhlargesymbols}{OMX}{yhex}{m}{n}
\DeclareMathAccent{\overarc}{\mathord}{yhlargesymbols}{"F3}

\makeatletter
\def\moverlay{\mathpalette\mov@rlay}
\def\mov@rlay#1#2{\leavevmode\vtop{%
    \baselineskip\z@skip \lineskiplimit-\maxdimen
    \ialign{\hfil$\m@th#1##$\hfil\cr#2\crcr}}}
\newcommand{\charfusion}[3][\mathord]{
  #1{\ifx#1\mathop\vphantom{#2}\fi
    \mathpalette\mov@rlay{#2\cr#3}
  }
  \ifx#1\mathop\expandafter\displaylimits\fi}
\DeclareRobustCommand\bigop[1]{%
  \mathop{\vphantom{\sum}\mathpalette\bigop@{#1}}\slimits@
}
\newcommand{\bigop@}[2]{%
  \vcenter{%
    \sbox\z@{$#1\sum$}%
    \hbox{\resizebox{\ifx#1\displaystyle.9\fi\dimexpr\ht\z@+\dp\z@}{!}{$\m@th#2$}}%
  }%
}
\makeatother
\newcommand{\bigjoin}{\bigop{\triangledown}}
\DeclareMathOperator{\join}{\triangledown}
\newcommand{\cupdot}{\charfusion[\mathbin]{\cup}{\cdot}}
\DeclareMathOperator{\bigcupdot}{\charfusion[\mathop]{\bigcup}{\cdot}}
\definecolor{jade}{rgb}{0.0, 0.66, 0.42}
\newcommand{\child}{\mathsf{child}}
\DeclareMathOperator*{\argmax}{arg\,max}
\newcommand{\Pmax}{\mathrm{Pmax}}
\newcommand{\T}{\widetilde{T}}
\renewcommand{\t}{\widetilde{t}}

\newcommand{\AX}[1]{\textnormal{#1}}

\providecommand{\keywords}[1]{\textbf{\textit{Keywords: }} #1}

\title{Hierarchical and Modularly-Minimal  Vertex Colorings}

\author[1]{Dulce I. Valdivia}

\author[2,3]{Manuela Gei{\ss}}

\author[4]{Maribel Hern{\'a}ndez Rosales}

\author[2,5-9]{Peter F. Stadler}    

\author[10,*]{Marc Hellmuth} 

\affil[1]{Departamento de Ingenier{\'i}a Gen{\'e}tica, Centro de Investigaci{\'o}n y de Estudios Avanzados del IPN (CINVESTAV), 
Km. 9.6 Libramiento Norte Carretera Irapuato-Le{\'o}n, MX-36821, Irapuato, GTO, M{\'e}xico}

\affil[2]{Bioinformatics Group, Department of Computer Science \&
  Interdisciplinary Center for Bioinformatics, Universit{\"a}t Leipzig,
  H{\"a}rtelstra{\ss}e 16-18, D-04107 Leipzig, Germany}

\affil[3]{Software Competence Center Hagenberg GmbH, Softwarepark 21, A-4232 Hagenberg, Austria}

\affil[4]{CONACYT-Instituto de Matem{\'a}ticas, UNAM Juriquilla, Blvd.\
  Juriquilla 3001, 76230 Juriquilla, Quer{\'e}taro, QRO, M{\'e}xico}

\affil[5]{German Centre for Integrative Biodiversity Research
  (iDiv) Halle-Jena-Leipzig, Competence Center for Scalable Data Services
  and Solutions Dresden-Leipzig, Leipzig Research Center for Civilization
  Diseases, and Centre for Biotechnology and Biomedicine at Leipzig
  University at Universit{\"a}t Leipzig}

\affil[6]{Max Planck Institute for Mathematics in the Sciences,
  Inselstra{\ss}e 22, D-04103 Leipzig, Germany} 

\affil[7]{Institute for Theoretical Chemistry, University of Vienna,
  W{\"a}hringerstrasse 17, A-1090 Wien, Austria}

\affil[8]{Facultad de Ciencias, Universidad National de Colombia, Sede
  Bogot{\'a}, Colombia}

\affil[9]{Santa Fe Insitute, 1399 Hyde Park Rd., Santa Fe NM 87501,
  USA}

\affil[10]{School of Computing, University of Leeds, EC Stoner
  Building, Leeds LS2 9JT, UK}

\affil[*]{corresponding author, email \texttt{mhellmuth@mailbox.org}}

\date{\ }

\begin{document}

\maketitle

\abstract{ 
  Cographs are exactly the hereditarily well-colored graphs, i.e., the graphs for which a greedy vertex coloring of every induced subgraph uses only the minimally necessary number of colors $\chi(G)$. We show that greedy colorings are a special case of the more general hierarchical vertex colorings, which recently were introduced in phylogenetic combinatorics. Replacing cotrees by modular decomposition trees generalizes the concept of hierarchical colorings to arbitrary graphs. We show that every graph has a modularly-minimal coloring $\sigma$ satisfying $|\sigma(M)|=\chi(M)$ for every strong module $M$ of $G$. This, in particular, shows that modularly-minimal colorings provide a useful device to design efficient coloring algorithms for certain hereditary graph classes.  For cographs, the hierarchical colorings coincide with the modularly-minimal coloring.  As a by-product, we obtain a simple linear-time algorithm to compute a modularly-minimal coloring of $P_4$-sparse graphs.
}

%
\section{Introduction}

Graph coloring problems appear as a natural formalization of diverse
real-life applications, describing in essence a partitioning of objects
into classes under a given set of constraints \cite{Jensen:95,Lewis:16}.
In this contribution, we investigate a specific type of vertex coloring
that naturally appear in computational biology. A detailed knowledge of the
evolutionary history of genes or species
\cite{DBH-survey05} is a prerequiste to answering many research questions
in biology.  In brief, the genome of an organism can be thought of as a
collection of genes. All organisms that belong to the same species share
the same collection of genes. Throughout evolution, species evolve
independently of each other and occasionally subdivide to form new
species. During this process of species-level evolution, also the genes
within a species' genomes change, and are sometimes lost or
duplicated. Since only those genes residing in species that are alive at
the present time can be observed and analyzed, the true evolutionary
history cannot be observed directly and hence must be inferred, using
algorithmic and statistical methods, from the genomic data available
today. A question of considerable practical importance is to decide whether
a pair of genes $x$ in species $A$ and $y$ in species $B$ are orthologs,
i.e., originated in a speciation event, or paralogs, i.e., were produced by
a gene duplication between speciation events
\cite{Fitch:70,TKL:97,GK13}. Since the true history is unknown, orthologous
gene pairs have to be distinguished from paralogs pairs using sequence
similarity as a measure of evolutionary relatedness.  A large class of
methods to determine orthology starts from so-called pairwise best hits
$\{x,y\}$, that is, of all genes in species $A$, the gene $x$ is most
similar to $y$, and of all genes in $B$, $y$ is most similar to $x$
\cite{Altenhoff:16}. This defines a graph $G$ on the set of genes. A
coloring $\sigma$ then assigns to each gene the species in which it
resides. A key result of \cite{Geiss:19b} is that if the edges of $G$
correctly represent orthology, then $(G,\sigma)$ is a so-called
hierarchically-colored cograph (a restricted types of colorings in graphs
that do not contain induced paths on four vertices).  The requirement of an
hierarchical coloring substantially strengthens the previously known
necessary condition that $G$ must be a cograph \cite{Hellmuth:13a}.

In this contribution we first study the properties of hierarchical
colorings and their relationship with greedy colorings of cographs. In
particular, we show that a coloring of a cograph is a greedy coloring if
and only if it is hierarchical w.r.t.\ all of its binary cotrees. On the
other hand, a coloring is minimal on each intermediate step along a binary
cotree if and only if it is hierarchical w.r.t.\ the same binary cotree.
These results motivate concepts of hierarchical and modularly-minimal
colorings of arbitrary graphs that are defined in terms of the modular
decomposition. As a main result we show that every graph $G$ has a
modularly-minimal coloring $\sigma$, that is, the subgraph $G[X]$ induced
by any strong module $X$ of $G$ is minimally colored.  As a by-product, we
obtain a simple linear-time algorithm to compute a
modularly-minimal coloring of $P_4$-sparse graphs in polynomial
time.

\section{Cographs and their Hierarchical Colorings}

Let $G=(V,E)$ be an undirected graph. A (proper vertex) coloring of $G$ is
a surjective map $\sigma: V\to S$ such that $xy\in E$ implies
$\sigma(x)\ne \sigma(y)$. We will often refer to such coloring as an
$|S|$-coloring.  The minimum number $|S|$ of colors such that there is an
$|S|$-coloring of $G$ is known as the \emph{chromatic number} $\chi(G)$.
For a subset $W\subseteq V$ (resp., a subgraph $H$ of $G$) we denote with
$\sigma(W)$ (resp., $\sigma(H)$) the set of colors assigned to the vertices
in $W$ (resp.\ $V(H)$) using $\sigma$.

A \emph{greedy coloring} of $G$ is obtained by ordering the set of colors
and coloring the vertices of $G$ in a random order with the first available
color. The \emph{Grundy number} $\gamma(G)$ is the maximum number of colors
required in a greedy coloring of $G$ \cite{Christen:79}. Obviously
$\gamma(G)\ge\chi(G)$. Determining $\chi(G)$ \cite{Karp:72} and $\gamma(G)$
\cite{Zaker:06} for arbitrary graphs are NP-complete problems.  A graph $G$
is called \emph{well-colored} if $\chi(G)=\gamma(G)$ \cite{Zaker:06}. It is
\emph{hereditarily well-colorable} if every induced subgraph is
well-colorable. For later reference, we provide the following useful 
\begin{lemma}\label{lem:CC-greedy}
  Let $\sigma$ be a greedy coloring of a disconnected graph $G$ with
  connected components $G_1,\dots, G_k$, $k\geq 2$ and let
  $H=\bigcupdot_{i\in I} G_i$ for some nonempty subset
  $I\subseteq\{1,\dots,k\}$. Then the restriction of $\sigma$ to $H$ is a
  greedy coloring of $H$.
\end{lemma}
\begin{proof}
  W.l.o.g.\ let the color set $\{1,\dots,|\sigma(G)|\}$ be naturally
  ordered from small to large integers.  Since $\sigma$ is a greedy
  coloring it necessarily colors every connected component $G_i$ with
  colors $\{1,\dots,|\sigma(G_i)|\}$.  Moreover, let us preserve the
  ordering $\prec_i$ on the vertices in each $G_i$ according to the order
  they are visited during the greedy coloring in $G$.  It is easy to verify
  that the first available color in $G$ to color a vertex $x$ in $G_i$ is
  precisely the first available to color $x$ when using the greedy coloring
  w.r.t. $\prec_i$ in $G_i$ only.  In other words, the greedy coloring can
  be applied independently on the connected components, which completes the
  proof.
\end{proof}

\begin{definition}[\cite{Corneil:81}]
  \label{def:Corneil}
  A graph $G$ is a \emph{cograph} if $G=K_1$, $G$ is the disjoint union
  $G=\bigcupdot_i G_i$ of cographs $G_i$, or $G$ is a join
  $G=\bigjoin_i G_i$ of cographs $G_i$.
\end{definition}
Since both operations, $\bigcupdot_i$ and $\bigjoin_i$, are commutative and
associative, each cograph can be written as the the join or disjoint union
of two cographs.  This recursive construction induces a rooted binary tree
$T$, whose leaves are individual vertices corresponding to a $K_1$ and
whose interior vertices correspond to the union and join operations. We
write $L(T)$ for the leaf set and $V^0(T)\coloneqq V(T)\setminus L(T)$ for
the set of inner vertices of $T$. The set of children of $u$ is denoted by
$\child(u)$. For edges $e=uv$ in $T$ we adopt the convention that $v$ is a
child of $u$. We define a labeling function $t:V^0(T)\rightarrow\{0,1\}$,
where an interior vertex $u$ of $T$ is labeled $t(u)=0$ if it is associated
with a disjoint union, and $t(u)=1$ for joins.  We will refer to $(T,t)$ as
a \emph{cotree}.  The tree $T(u)$ denotes the subtree of $T$ that is rooted
at $u$. To simplify the notation we will write $G(u):=G[L(T(u))]$ for the
subgraph of $G$ induced by the vertices in $L(T(u))$. Note that $G(u)$ is
the graph consisting of the single vertex $u$ if $u$ is a leaf of
$T$. Furthermore, $G(u)$ is a cograph by definition.

Given a cograph $G$, there is a unique \emph{discriminating}
cotree\footnote{In \cite{Corneil:81} the discriminating cotree is defined
  as \emph{the} cotree associated with $G$. Here we call every tree $(T,t)$
  a cotree as it is always a ``refinement'' of some discriminating cotree
  that explain the same cograph \cite{Boecker:98}.}  $(T^*,t^*)$ in which
adjacent operations are distinct, i.e., $t^*(u)\neq t^*(v)$ for all
interior edges $uv\in E(T)$. The discriminating cotree $(T^*,t^*)$ is
obtained from every arbitrary binary cotree $(T,t)$ by contracting all
edges $uv$ with $t(u)=t(v)$ into a single vertex $\overarc{uv}$ with label
$t^*(\overarc{uv})=t(u)=t(v)$.  Conversely, every binary cotree of $G$ can
be obtained by replacing an inner vertex of $u$ and its children
$u_1,\dots, u_k$ by an arbitrary binary tree with root $u$, leaves
$u_1,\dots, u_k$ and all its inner vertices $w$ labeled by $t(w)=t^*(u)$,
see e.g.\ \cite{Boecker:98,Corneil:81} for details.

It is possible to refine the discriminating cotree by subdividing a
disjoint union or join into multiple disjoint unions or joins, respectively
\cite{Boecker:98,Corneil:81}. It is well known that every induced subgraph
of a cograph is again a cograph. A graph is a cograph if and only if it
does not contain a path $P_4$ on four vertices as an induced subgraph
\cite{Corneil:81}.  The cographs are also exactly the hereditarily
well-colored graphs \cite{Christen:79}. The chromatic number of a cograph
$G$ can be computed recursively, as observed in
\cite[Tab.1]{Corneil:81}. Starting from $\chi(K_1)=1$ as base case we have
\begin{equation}
  \begin{split}
    \chi(G) &= \chi\left(\bigcupdot_{i} G_i\right) = \max_{i} \chi(G_i)
    \textrm{ or }\\
    \chi(G) &= \chi\left(\bigjoin_{i} G_i\right) = \sum_{i} \chi(G_i)
  \end{split}
  \label{eq:rec1}
\end{equation}

\emph{Hierarchically colored cographs} (\emph{hc-cographs}) were introduced
in \cite{Geiss:19b} as the undirected colored graphs recursively defined by
 \begin{description}
 \item[\AX{(K1)}] $(G,\sigma)=(K_1,\sigma)$, i.e., a colored vertex, or
 \item[\AX{(K2)}] $(G,\sigma)= (H_1,\sigma_{H_1}) \join (H_2,\sigma_{H_2})$
   and $\sigma(V(H_1))\cap \sigma(V(H_2))=\emptyset$, or
 \item[\AX{(K3)}]
   $(G,\sigma)= (H_1,\sigma_{H_1}) \cupdot (H_2,\sigma_{H_2})$ and 
   $\sigma(V(H_1))\cap \sigma(V(H_2)) \in
   \{\sigma(V(H_1)),\sigma(V(H_2))\}$,
\end{description}
where $\sigma(x) = \sigma_{H_i}(x)$ for every $x\in V(H_i)$, $i\in\{1,2\}$
and $(H_1,\sigma_{H_1})$ and $(H_2,\sigma_{H_2})$ are hc-cographs.

Obviously, the graph $G$ underlying an hc-cograph is a cograph. Thus, the
recursive construction of an hc-cograph $G$ implies a binary cotree $(T,t)$
that can be constructed with a top down approach as follows: Denote the
root of $(T,t)$ by $r$. It is associated with the graph $G(r)=G$. In the
general step we consider an induced subgraph $G(u)$ of $G$ associated with
a vertex $u$ of $T$. If $G(u)$ is connected, then $t(u)=1$ and $G(u)$ is
the joint of pair of induced subgraphs $G(v_1)$ and $G(v_2)$. To identify
these graphs, consider the connected components
$\overline{G_1},\dots,\overline{G_k}$ of the complement $\overline{G(u)}$
of $G(u)$. We have
\begin{equation}
  G(u) = \overline{\bigcupdot_{i=1}^k \overline{G_i}} =
  \bigjoin_{i=1}^k \overline{\overline{G_i}} = 
  \bigjoin_{i=1}^k G_i \,.
\end{equation}
We therefore set $G(v_1)=G_1$ and
$G(v_2)=\overline{\bigcupdot_{i=2}^k \overline{G_i}}=\bigjoin_{i=2}^k
G_i$. By construction, we therefore have $G(u)=G(v_1)\join G(v_2)$ with
disjoint color sets $\sigma(G(v_1))$ and $\sigma(G(v_2))$. 
If $G(u)$ is disconnected, define
$t(u)=0$, identify one of the components, say $G_1$, with the smallest
numbers of colors $|\sigma(G_1)|$ and set $G(v_1)=G_1$ and
$G(v_2)=G(u)\setminus G(v_1)$. The fact that $G(u)$ is an hc-cograph
ensures that $\sigma(G(v_1))\subseteq \sigma(G(v_2))$. In both the
connected and the disconnected case we attach $v_1$ and $v_2$ as the
children of $u$ in $T$. The reconstruction of $(T,t)$ can be performed in
linear time.  

\begin{definition}\label{def:hc-coloring}
  A coloring $\sigma$ of a cograph $G$ is a \emph{hierarchical coloring
    w.r.t.\ the binary cotree $(T,t)$} if $(T,t)$ is a cotree of $G$ and
  $(G,\sigma)$ is a hc-cograph recursively constructed according to
  $(T,t)$.
\end{definition}
As noticed in \cite{Geiss:19b}, a coloring $\sigma$ of a cograph $G$ may be
hierarchical w.r.t.\ a binary cotree $(T,t)$ but not hierarchical w.r.t.\
another binary cotree $(T',t')$ that yields the same cograph. An example
is shown in Fig.\ \ref{Fig:diffcotree}.

\begin{figure}[h]
  \begin{center}
    \includegraphics[scale=0.55]{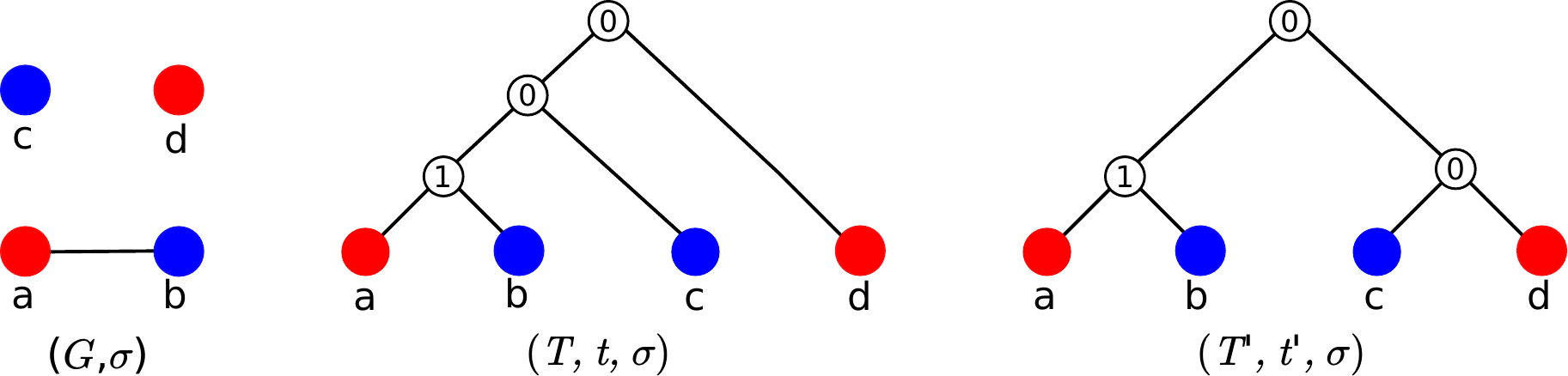}		
  \end{center}
  \caption{The induced cotree of a cograph $G$ affects the hierarchical
    coloring property of $\sigma$, where
    $\sigma(a)=\sigma(d)\neq\sigma(b)=\sigma(c)$. In $(T,t)$, the first
    tree from left to right, \AX{(K1)}-\AX{(K3)} are satisfied making
    $\sigma$ hierarchical.  However the second tree $(T',t')$ does not
    satisfy \AX{(K3)} since the parent node of $c\simeq K_1$ and
    $d\simeq K_1$ corresponds to a disjoint union operation and
    $\sigma(c)\cap\sigma(d)=\emptyset$. Thus $\sigma$ is not hierarchical
    w.r.t. $(T',t')$.}
  \label{Fig:diffcotree}
\end{figure}

Every hierarchical coloring of a cograph is also a proper coloring (cf.\
\cite[Lemma 43]{Geiss:19b}).  The property of being a cograph is hereditary
(i.e., every induced subgraph of a cograph is a cograph). However, this is
not necessarily true for hc-cographs. As an example consider the induced
disconnected subgraph with vertices $c$ and $d$ of the hc-cograph in Fig.\
\ref{Fig:diffcotree} that violates $\AX{(K3)}$. Nevertheless, if $G$ is an
hc-cograph, then each of its connected components must be an hc-cograph
(cf.\ \cite[Lemma 44]{Geiss:19b}).  We show now that hc-cographs are always
optimally colored.
\begin{theorem} 
  Let $\sigma$ be a hierarchical coloring of a cograph $G$ w.r.t.\ some
  binary cotree $(T,t)$. Then $|\sigma(V)|=\chi(G)$.
  \label{thm:hc-implies-chrom}
\end{theorem}
\begin{proof}
  We proceed by induction w.r.t.\ $|V|$. The statement is trivially true
  for $|V|=1$, i.e. $G=K_1$, since $\chi(K_1)=1$. Now suppose
  $|V|>1$. Thus, $G=G_1\join G_2$ or $G=G_1 \cupdot G_2$ for some
  hc-cographs $G_1=(V_1,E_1)$ and $G_1=(V_2,E_2)$ with
  $1\leq|V_1|,|V_2|<|V|$. By induction hypothesis we have
  $|\sigma(V_1)|=\chi(G_1)$ and $|\sigma(V_2)|=\chi(G_2)$.

  First consider $G=G_1\join G_2$. Since $xy\in E(G)$ for all $x\in V_1$
  and $y\in V_2$ we have $\sigma(x)\ne\sigma(y)$, and hence
  $\sigma(V_1)\cap\sigma(V_2)=\emptyset$. Thus,
  $\sigma(V)=\sigma(V_1)\cupdot\sigma(V_2)$ and therefore,
  \begin{equation*}
    |\sigma(V)| = |\sigma(V_1)| + |\sigma(V_2)| \stackrel{\text{IH.}}{=}
    \chi(G_1)+\chi(G_2) \stackrel{\text{Equ.\ }\eqref{eq:rec1}}{=} \chi(G).
  \end{equation*}
  We note that the coloring condition in \AX{(K2)} therefore only enforces
  that $\sigma$ is a proper vertex coloring.

  Now suppose $G=G_1\cupdot G_2$. Axiom \AX{(K3)} implies
  $|\sigma(V)|=\max( |\sigma(V_1)|,|\sigma(V_2)|)$. Hence,
  \begin{equation*}          
    |\sigma(V)| = \max( |\sigma(V_1)|,|\sigma(V_2)|)
    \stackrel{\text{IH.}}{=}
    \max(\chi(G_1),\chi(G_2)) \stackrel{\text{Equ.\ }\eqref{eq:rec1}}{=}
    \chi(G).
  \end{equation*}
\end{proof}

Since the connected components of hc-cographs are again hc-cographs, the
following statement is an immediate consequence of the recursive
construction of hc-cographs:
\begin{corollary}
  \label{cor:chi(G(u))}
  If $\sigma$ is a hierarchical coloring of $G$ w.r.t.\ the binary
  cotree $(T,t)$, then $|\sigma(G(u))|=\chi(G(u))$ for all nodes $u$ of
  $(T,t)$ and $|\sigma(G')|=\chi(G')$ for all connected components $G'$ of
  $G$.
\end{corollary}

As detailed in \cite{Christen:79}, we have $\chi(G)=\gamma(G)$ for
cographs. Thus, it seems natural to ask whether every greedy coloring is
hierarchical w.r.t.\ some binary cotree $(T,t)$. Making use of the fact
that $\chi(G)=\gamma(G)$, we assume w.l.o.g.\ that the color set is
$S=\{1,2,\dots,\chi(G)\}$ whenever we consider greedy colorings of a
cograph.

We shall say that a cograph $G$ is a \emph{minimal counterexample for some
  property $\mathcal{P}$} if (1) $G$ does not satisfy $\mathcal{P}$ and (2)
every induced subgraph of $G$ (i.e., every ``smaller'' cograph) satisfies
$\mathcal{P}$.

\begin{lemma}
  \label{lem:grehc}
  Let $G$ be a cograph, $(T,t)$ an arbitrary binary cotree for $G$ and
  $\sigma$ a greedy coloring of $G$. Then $\sigma$ is a hierarchical
  coloring w.r.t.\ $(T,t)$.
\end{lemma}
\begin{proof}
  Assume $G$ is a minimal counterexample, i.e., $G$ is a minimal cograph
  for which a coloring $\sigma$ exists that is a greedy coloring but not an
  hierarchical coloring. If $G$ is connected, then either $G\simeq K_1$ or
  $G=\bigjoin_{i=1}^n G_i$ and $\sigma(V)=\bigcupdot_{i=1}^n \sigma(V_i)$
  for some $n>1$, i.e., \AX{(K1)} or \AX{(K2)} is satisfied.  By
  assumption, $\sigma$ is not a hierarchical coloring, hence $\sigma$ must
  fail to be a hierarchical coloring on at least one of the subgraphs
  $G_i$, contradicting the assumption that $G$ is a minimal
  counterexample. Thus, $G$ cannot be connected.

  Therefore, assume that $G=\bigcupdot_{i=1}^n G_i$ for some $n>1$. Since
  $G$ is represented by a binary cotree $(T,t)$, the root of $T$ must have
  exactly two children $u$ and $v$. Hence, we can write
  $G=G(u)\cupdot G(v)$.  Since $G$ is a minimal counterexample and since,
  by Lemma \ref{lem:CC-greedy}, $\sigma$ restricted to $G(u)$, resp.,
  $G(v)$ is a greedy coloring of $G(u)$, resp., $G(v)$, we can conclude
  that $\sigma$ induces a hierarchical coloring on $G(u)$ and $G(v)$.
  However, since $\sigma$ is, in particular, a greedy coloring of $G(u)$
  and $G(v)$, $\sigma(V(G(u)))\subseteq \sigma(V(G(v)))$ or
  $\sigma(V(G(v)))\subseteq \sigma(V(G(u)))$ most hold. But this
  immediately implies that $(G,\sigma)$ satisfies $\AX{(K3)}$ and thus
  $\sigma$ is a hierarchical coloring of $G$.  Therefore, $G$ is not a
  minimal counterexample, which completes the proof.
\end{proof}

Since every cograph is well-colorable we obtain as an immediate consequence
\begin{corollary}\label{cor:allCog}
  Every cograph has a hierarchical coloring w.r.t.\ each of its binary
  cotrees $(T,t)$.
\end{corollary}
The converse of Lemma \ref{lem:grehc} is not true.  Fig.\
\ref{Fig:hcnogrundy} shows an example of a hierarchical coloring that is
not a greedy coloring.

\begin{figure} 
  \begin{tabular}{ccc}
    \begin{minipage}{0.3\textwidth}
      \includegraphics[width=\textwidth]{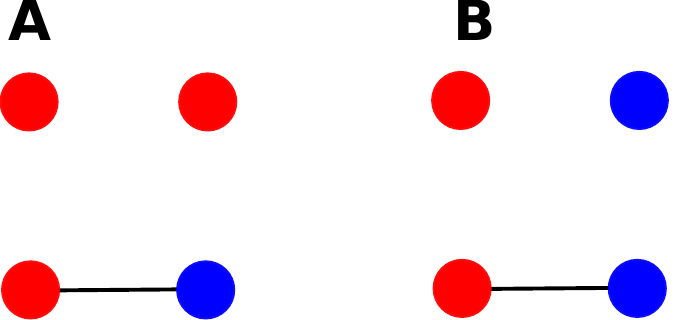}
    \end{minipage} &&
    \begin{minipage}{0.6\textwidth}
      \caption{Both colorings \textsf{A} and \textsf{B} of
        $K_2\cupdot K_1\cupdot K_1$ are hierarchical colorings w.r.t.\
        some cotrees. To see this, note that the coloring \textsf{A} is
        produced by $(K_1\join K_1)\cupdot (K_1\cupdot K_1)$, while
        coloring \textsf{B} is the result of
        $((K_1\join K_1)\cupdot K_1)\cupdot K_1$.  Only \textsf{A} is a
        greedy coloring. The coloring \textsf{B} uses different colors for
        the two single-vertex components and thus is not greedy.}
    \end{minipage} 
  \end{tabular}
  \label{Fig:hcnogrundy}
\end{figure}

\begin{theorem}
  A coloring $\sigma$ of a cograph $G$ is a greedy coloring if and only if
  it is a hierarchical coloring w.r.t.\ every binary cotree $(T,t)$ of $G$.
  \label{thm:greedy}
\end{theorem}
\begin{proof}
  By Lemma~\ref{lem:grehc}, every greedy coloring is a hierarchical
  coloring for every binary cotree $(T,t)$. Now suppose $\sigma$ is an
  hierarchical coloring for every binary cotree $(T,t)$ and let $G$ be a
  minimal cograph for which $\sigma$ is not a greedy coloring. As in the
  proof of Lemma \ref{lem:grehc}, we can argue that $G$ cannot be a minimal
  counterexample if $G$ is connected: in this case,
  $G=\bigjoin_{i=1}^n G_i$ and $\sigma(V)=\bigcupdot_{i=1}^n \sigma(V_i)$
  for all colorings, and thus $\sigma$ is a greedy coloring if and only if
  it is a greedy coloring with disjoint color sets for each $G_i$. Hence, a
  minimal counterexample must have at least two connected components.

  Let $G=\bigcupdot_{i=1}^n G_i$ for some $n>1$ and define a partition of
  $\{1,\dots,n\}$ into sets $I_1,\dots,I_{\ell}$, $\ell\geq 1$, such that
  for every $r\in \{1,\dots,\ell\}$ we have $i,j\in I_r$ if and only if
  $\chi(G_i) = \chi(G_j)$. Since every
  $G^r \coloneqq \bigcupdot_{i\in I_r} G_i$, $1\leq r\leq\ell$, is a
  cograph, each $G^r$ can be represented by a (not necessarily unique)
  binary cotree $(T^r,t^r)$.  Note, we have $\chi(G^r)=\chi(G_i)$ for all
  $i\in I_r$. Now, we can construct a binary cotree $(T,t)$ for $G$ as
  follows: let $T^*$ be a caterpillar with leaf set
  $L(T^*)=\{l_1,\dots,l_\ell\}$.  We choose
  $T^* = (\ldots((l_1,l_2),l_3), \ldots l_\ell)$ (in Newick format). Note,
  if $\ell=1$, then $T^*\simeq K_1$.  Now, the root of every tree
  $T^1,\dots,T^r$ is identified with a unique leaf in $L(T^*)$ such that
  the root of $T_i$ is identified with $l_i\in L(T^*)$ and the root of
  $T_j$ is identified with $l_j\in L(T^*)$, where $i<j$ if and only if
  $\chi(G^i)<\chi(G^j)$. This yields the tree $T$.  The labeling $t$ for
  $(T,t)$ is provided by keeping the labels of each $(T^r,t^r)$ and by
  labeling all other inner vertices of $T$ by $0$.  It is easy to see that
  $(T,t)$ is a binary cotree for $G$. By assumption, $\sigma$ is
  hierarchical w.r.t.\ $(T,t)$. We denote by $C(T^*)\subseteq V(T)$ the set
  of inner vertices of $T^*$. Since $\sigma$ is hierarchical w.r.t.\
  $(T,t)$ and thus in particular w.r.t.\ any subtree $(T^r,t^r)$, we have
  $\sigma(V(G_i))\cap \sigma(V(G_j))\in \{\sigma(V(G_i)),\sigma(V(G_j))\}$
  for any $i,j\in I_r$, $1\le r\le\ell$, by \AX{(K3)}. Hence, as
  $\chi(G^r)=\chi(G_i)=\chi(G_j)$, it must necessarily hold
  $\sigma(V(G_i))= \sigma(V(G_j))$ for all $i,j\in I_r$, i.e., all
  connected components $G_i$ with the same chromatic number are colored by
  the same color set. By construction, at every node $v\in C(T^*)$ of the
  caterpillar structure, with children $v'$ and $v''$, the components
  $G':=G(v')$ and $G'':=G(v'')$ satisfy $\chi(G')<\chi(G'')$.  Invoking
  \AX{(K3)} we therefore have $\sigma(G')\subset\sigma(G'')$ and
  $G^*:=G(v)$ is colored by the color set $\sigma(G^*)=\sigma(G'')$.  These
  set inclusions therefore imply a linear ordering of the colors such that
  colors in $\sigma(G')$ come before those in
  $\sigma(G'')\setminus \sigma(G')$. Thus $\sigma$ is a greedy coloring
  provided that the restriction of $\sigma$ to each of the connected
  components $G_i$ of $G$ is a greedy coloring, which is true due to the
  assumption that $G$ is a minimal counterexample. Thus no minimal
  counterexample exists, and the coloring $\sigma$ is indeed a greedy
  coloring of $G$.
\end{proof}

Not every minimal coloring of a cograph is hierarchical w.r.t.\ some binary
cotree. For instance, if $G = G_1\bigcupdot G_2$ is the disjoint union of
two connected graphs $G_1$ and $G_2$, then it suffices that
$\chi(G_1)<|\sigma(V_1)| = \chi(G)$. That is, we may use more colors than
necessary on $G_1$. Note that for every cotree $(T,t)$ of $G$ there is
  a vertex $u$ in $T$ such that $G_1=G(u)$.  Hence, for every cotree we
  have $\chi(G(u))<|\sigma(G(u))|$ with $G(u)=G_1$.  Contraposition of
Cor.\ \ref{cor:chi(G(u))} implies that $\sigma$ cannot be a hierarchical
coloring of $G$ w.r.t.\ any cotree $(T,t)$ of $G$.  In the following we
restrict our attention to those coloring that satisfy the necessary
conditions of Cor.\ \ref{cor:chi(G(u))}, which is specified in the next
 \begin{definition}
  Let $G$ be a cograph with cotree $(T,t)$. A coloring $\sigma$ is
  \emph{$(T,t)$-minimal} if $|\sigma(G(u))|=\chi(G(u))$ for every vertex
  $u$ of the cotree $(G,t)$.
\end{definition}

\begin{theorem}
  A coloring $\sigma$ of cograph $G$ is $(T,t)$-minimal for a binary cotree
  $(T,t)$ if and only it is hierarchical w.r.t.\ $(T,t)$.
  \label{thm:mrc}
\end{theorem}
\begin{proof}
  By Corollary \ref{cor:chi(G(u))}, $\sigma$ is $(T,t)$-minimal if it is
  hierarchical w.r.t.\ $(T,t)$.

  Now suppose there is a minimal cograph $G$ with a coloring $\sigma$ that
  is $(T,t)$-minimal but not hierarchical w.r.t.\ $(T,t)$.  If $G$ is
  connected, then $G=\bigjoin_{i=1}^n G_i$ for some $n\ge 2$ and the
  restrictions of $\sigma$ to the connected components $G_i$ use disjoint
  color sets. Hence, $\sigma$ is a hierarchical coloring whenever the
  restriction to each $G_i$ is a hierarchical coloring.  Thus a minimal
  counterexample cannot be connected. Now suppose
  $G=\bigcupdot_{i=1}^n G_i$ for some $n\ge2$.  Since $(G,\sigma)$ is by
  assumption a minimal counterexample, each connected component
  $(G_i,\sigma_i)$ is a hierarchical cograph (w.r.t.\ subtrees of
  $(T,t)$). Since $(T,t)$ is binary, its root $u$ has two children $u_1$
  and $u_2$ that correspond to $G_1 = G(u_1)$ and $G_2 = G(u_2)$, resp.,
  such that $G=G_1\cupdot G_2$.  By Equ.\ \eqref{eq:rec1}, we can
    choose the notation such that $\chi(G)=\chi(G_1)$.  By definition,
  $\sigma$ induces a coloring $\sigma_1$ on $G_1$ and $\sigma_2$ on $G_2$.
  Clearly, each coloring $\sigma_i$ is $(T(u_i),t_{|T(u_i)})$-minimal.
  Since $G$ is a minimal counterexample, $\sigma_1$ and $\sigma_2$ are both
  hierarchical colorings of $G_1$ and $G_2$, respectively, in other words
  $(G_1,\sigma_1)$ and $(G_2,\sigma_2)$ are hc-cographs w.r.t.\ subtrees of
  $(T,t)$ that are rooted at $u_1$ and $u_2$, respectively.  Moreover,
  $\chi(G)=\chi(G_1)$ implies $\sigma_2(V(G_2))\subseteq
  \sigma_1(V(G_1))$. In summary, therefore,
  $(G,\sigma) = (G_1,\sigma_1) \cupdot (G_2,\sigma_2)$ satisfies \AX{(K3)},
  thus it is a cograph with a hierarchical coloring $\sigma$.  Hence, there
  cannot exist a minimal cograph with a coloring $\sigma$ that is
  $(T,t)$-minimal but not a hierarchical coloring w.r.t.\ $(T,t)$.
\end{proof}

Given a cograph $G$ and its corresponding binary cotree $(T,t)$ it is not
difficult to construct a $(T,t)$-minimal coloring.

\begin{algorithm}
  \caption{$(T,t)$-minimal coloring a cograph $G$ with binary 
    cotree $(T,t)$.}
  \label{alg:cBMG2}
  \algsetup{linenodelimiter=}
  \begin{algorithmic}[1]
    \REQUIRE Cograph $G$ and binary cotree $(T,t)$
    \STATE initialize coloring $\sigma$ s.t.\ all $v \in V(G)$ 
           have different colors
    \FORALL[from bottom to top where each $u$ is processed after 
            all its children have been processed]{$u\in V^{0}(T)$}
       \IF {$t(u) = 0$} 
          \STATE Let $v,w$ be the children of $u$
          \STATE $G^* \leftarrow \argmax\{\chi(G(v)),\chi(G(w))\}$
          \STATE $S \leftarrow \sigma(V(G^*))$ 
          \STATE $H\leftarrow G(v)$ if $G^*=G(w)$, otherwise $H\leftarrow G(w)$
          \STATE randomly choose an injective map $\phi:\sigma(H)\to S$ \label{alg:cBMG2:recol1}
          \FORALL {$x\in H$}
             \STATE $\sigma(x)\leftarrow \phi(\sigma(x))$   \label{alg:cBMG2:recol2}
          \ENDFOR
       \ENDIF
    \ENDFOR
  \end{algorithmic}
\end{algorithm}

\begin{theorem}\label{Thm:Algo1}
  Given a cograph $G$ and its corresponding binary cotree $(T,t)$, Alg.\
  \ref{alg:cBMG2} returns a $(T,t)$-minimal coloring $\sigma$ in polynomial
  time.
\end{theorem}
\begin{proof} 
  We need to show that the Algorithm constructs a $(T(u),t)$-minimal
  coloring for every vertex $u$ of $T$. For the leaves this is trivial.  We
  claim that Alg.\ \ref{alg:cBMG2} correctly generates a $(T(u),t)$-minimal
  coloring at each inner vertex $u$ of $(T,t)$ provided the colorings at
  the two children $v$ and $w$ of $u$ are 
  $(T(v),t)$-minimal and $(T(w),t)$-minimal, respectively.  It is clear
  that the color sets $G(v)$ and $G(w)$ are disjoint while $u$ is
  processed. If follows immediately, therefore, that the coloring of $G(u)$
  is $(T(u),t)$-minimal if $t(u)=1$.  Thus, in the algorithm we can safely
  ignore this case.

  For $t(u)=0$, Alg.\ \ref{alg:cBMG2} determines the graph
  $G^*\in \{G(v),G(w)\}$ with the largest number of colors, say
  $\chi(G(v))=|\sigma(G(v))|\le|\sigma(G(w))|=\chi(G(w))$. Since
  $\chi(G(u))=\chi(G(w))$ we can color $G(u)$ with the color set
  $S=\sigma(G(w))$ of $G(w)$. To this end, we recolor $G(v)$ with an
  injective map $\phi:\sigma(G(v)) \rightarrow S$. Such a map exists since
  $|\sigma(G(v))|=\chi(G(v))\le|S|=\chi(G(w))$.  After recoloring
  $|\sigma(G(v))|=|S|=\chi(G(w))=\chi(G(u))$. Thus, the resulting coloring
  of $G(u)$ is again $(T(u),t)$-minimal.  Since each leave $v$ is trivially
  $(T(v),t)$-minimally colored, we conclude that Alg.\ \ref{alg:cBMG2} is
  correct and can clearly be implemented to run in polynomial-time.
\end{proof}

\begin{corollary}
  For every cograph $G$ and every cotree $(T,t)$ there is a $(T,t)$-minimal
  coloring.
\end{corollary}

The recursive structure of hc-cographs can also be used to count the number
of distinct hierarchical colorings of $G$ w.r.t.\ a given binary cotree
$(T,t)$. For an inner vertex $u$ of $T$ denote by $Z(G(u))$ the number of
hc-colorings of $G(u)$. If $u$ is a leaf, then $Z(u)=1$, otherwise, $u$ has
exactly two children, $\child(u)=\{v_1,v_2\}$.  For $t(u)=1$, we have
$Z(G(u))=Z(G(v_1)) \cdot Z(G(v_2))$ since the color sets are disjoint. If
$t(u)=0$, assume, w.l.o.g.\
$s_1:=|\sigma(G(v_1))|\le|\sigma(G(v_2))|=:s_2$,
$Z(G(u))=Z(G(v_1))\cdot Z(G(v_2)) \cdot g(s_1,s_2)$, where $g(s_1,s_2)$ is
the number of injections between a set of size $s_1$ into a set of size
$s_2$, i.e., $g(s_1,s_2)=\binom{s_2}{s_1} s_1!$. The number of coloring can
now be computed by bottom-up traversal on $T$.

The total number of hc-colorings can be obtained by considering a
caterpillar tree for the step-wise union of connected components. For each
connected component $G_i$ with $\chi(G_i)=s_i$, and $s=\max_i s_i$ there
are $\binom{s}{s_i}$ choices of the colors, i.e., $g(s,s_i)$ injections and
thus $Z(G)=\prod_i (g(s,s_i)\cdot Z(G_i))$ colorings. We note in passing
that the chromatic polynomial of a cograph, and thus the number of
colorings using the minimal number of colors, can also be computed in
polynomial time \cite{Makowsky:06}. There does not seem to be an obvious
connection between the hierarchical colorings and the chromatic polynomial
of a cograph, however.

\section{Modularly-Minimal Colorings}

The definition of hierarchical colorings in the previous section crucially
depends on the structure of cographs and their associated cotrees.  In
order to extend the concept to arbitrary graphs, we first need some
additional notation. We denote the neighborhood of a vertex $v\in V$ by
$N(v)$ and recall
\begin{definition}[\cite{Gallai:67}] 
  Let $G=(V,E)$ be an arbitrary graph. A non-empty vertex set
  $X\subseteq V$ is a \emph{module} of $G$ if, for every
  $y\in V\setminus X$, either $N(y)\cap X=\emptyset$ or $X\subseteq N(y)$
  is true. A module $M$ is \emph{strong} if it is does not overlap with any
  other module $M'$, i.e., if $M\cap M'\in \{M,M',\emptyset\}$.
\label{def:module}
\end{definition}
In particular $V$ and the singletons $\{v\}$, $v\in V$ are strong modules.
The \emph{maximal modular partition} of a graph $G=(V,E)$ with $|V|>1$,
denoted by $\Pmax(G)=\{M_1,\dots,M_k\}$, is a partition of the vertex set
$V$ into inclusion-maximal strong modules distinct from $V$.  In
particular, if $G$ or $\overline{G}$ are disconnected, then the respective
connected components are the elements of $\Pmax(G)$.

The \emph{modular decomposition} \cite{Gallai:67} of a graph $G$ is based
on $\Pmax(G)$ and recursively decomposes $G$ into strong modules in
$\Pmax(G)$.  This recursive decomposition of $G$ corresponds to the modular
decomposition (MD) tree of $G$, that is, a vertex-labeled tree $(\T,\t)$
where each of its vertices is associated with a strong module $X$ of $G$
and a label $\t$ that distinguishes the three cases: (i) \emph{parallel}:
the induced subgraph of $G$ by $X$, $G[X]$, is disconnected, (ii)
\emph{series}: $\overline{G[X]}$ is disconnected, and (iii) \emph{prime}:
both $G[X]$ and $\overline{G[X]}$ are connected. We write $\mathcal{M}(G)$
for the set of strong modules of $G$.  The maximal modular partition
$\Pmax(G)$, the MD tree, and the set of strong modules $\mathcal{M}(G)$ of
$G$ are unique \cite{HP:10}. The modular decomposition of $G$, and thus its
set of strong modules, can be obtained in linear time \cite{McConnell:99}.

We first note for later reference that all proper colorings necessarily
satisfy a generalization of \AX{(K2)} for series nodes of the MD tree. By
abuse of notation, we will call a node $u$ of the MD tree $(\T,\t)$
parallel, series, or prime if the corresponding vertex set $L(\T(u))$ is a
parallel, series or prime module of $G$.
\begin{lemma}
  Let $\sigma$ be a proper coloring of a graph $G$ and let $X$ be a strong
  series module of $G$ with $\Pmax(G[X])=\{M_1,\dots,M_k\}$. Then
  $\sigma(M_i)\cap\sigma(M_j)=\emptyset$ whenever $i\ne j$.
  \label{lem:K2gen}
\end{lemma}
\begin{proof}
  Since the $M_i$ are strong modules of $X$, each $x\in M_i$ is adjacent to
  every $y\in X\setminus M_i$. Since $\sigma$ is a proper coloring, we have
  $\sigma(x)\ne\sigma(y)$, i.e., no color appearing in $M_i$ can appear
  elsewhere in $X$.
\end{proof}

A graph is a cograph if and only if all nodes in its MD tree are series or
parallel \cite{Corneil:81}. Since a cograph is either a $K_1$ or it can be
written as $G=G_1\cupdot G_2$ or $G=G_1\join G_2$, both $G_1$ and $G_2$ are
modules of $G$. It immediately follows that for every binary cotree
$(T,t)$, the vertex sets $L(T(u))=V(G(u))$ are modules of $G$ for all
vertices $u$ in $T$. In general, however, these modules are not strong.

\begin{lemma}
  \label{lem:allStrongInCotree}
  Let $(T,t)$ be a binary cotree of a cograph $G$.  Then, every strong
  module of $G$ is an induced subgraph $G(u)$ for some vertex $u$ in $T$.
\end{lemma}
\begin{proof}
  Let $M$ be a strong module of $G$ and assume, for contradiction, that
  $M\ne V(G(u))=L(T(u))$ for all $u$ in $T$. Then there is a vertex $v$ in
  $T$ such that $M\subsetneq V(G(v))$ and $V(G(v))$ is an inclusion-minimal
  set belonging to a node $v$ of $T$ containing $M$.  Let $v_1$ and $v_2$
  be the two children of $v$ in $T$. By construction, $M$ intersects both
  $V(G(v_1))$ and $V(G(v_2))$, and is properly contained in their
  union $V(G(v))$.  However, since $M$ is strong, it can neither overlap
  $V(G(v_1))$ nor $V(G(v_2))$ and thus, $V(G(v_1)), V(G(v_2))\subseteq M$.
  Therefore, $M= V(G(v_1))\cup V(G(v_2)) = V(G(v))$; a contradiction.  Thus,
  $M=V(G(u))$ for some vertex $u$ in $T$.
\end{proof}

The discriminating cotree $(T^*,t^*)$ of a cograph $G$ coincides with its
modular decomposition tree $(\T,\t)$ \cite{Corneil:81}. 
Together with Lemma
\ref{lem:K2gen}, this suggests to generalize the concept of hierarchical
colorings to arbitrary graphs.
\begin{definition}
  A coloring $\sigma$ of $G$ is \emph{hierarchical} if, for every
  disconnected strong module $G[X]$ of $G$ there is a strong module
  $M_j\in\Pmax(G[X])$ such that $\sigma(M_i)\subseteq\sigma(M_j)$ for all
  $M_i\in\Pmax(G[X])$.
  \label{def:genhier}
\end{definition}

Property \AX{(K3)}, furthermore suggests a stronger variant:
\begin{definition}
  A coloring $\sigma$ of $G$ is \emph{strictly hierarchical} if for every
  disconnected strong  module $G[X]$ of $G$ we have
  $\sigma(M_i)\cap \sigma(M_j)\in \{\sigma(M_i),\sigma(M_j)\}$ for all
  $M_i,M_j\in\Pmax(G[X])$.
  \label{def:genhierstrict}
\end{definition} 

If $\sigma$ is a strictly hierarchical coloring, then a module $M_j$ whose
color set has maximum size, satisfies $\sigma(M_i)\subseteq\sigma(M_j)$ for
all $i\in\{1,\dots,k\}$. Thus strictly hierarchical implies hierarchical.
The converse, however, is not always true, since
$\sigma(M_i)\subseteq\sigma(M_j)$ for all $M_i\in\Pmax(G[X])$ does not
prevent the modules distinct from $M_j$ from having overlapping color sets.

\begin{theorem}
  Every graph $G$ has a hierarchical and a strictly hierarchical
  coloring.
  \label{thm:strhier}
\end{theorem}
\begin{proof}
  Since strictly hierarchical implies hierarchical, it suffices to show
  that $G=(V,E)$ has a strictly hierarchical coloring. We proceed with
  induction on $|V|$.  Clearly, the single vertex graph $K_1$ has a
  strictly hierarchical coloring.  Now suppose that every graph with
  less than $|V|$ vertices has a strictly hierarchical coloring.  Let
  $\Pmax(V)=\{M_1,\dots,M_k\}$.  By induction hypothesis, each $G[M_i]$ has
  a strictly hierarchical coloring $\sigma_i$.
		
  If $V$ is a series or prime module, then the colorings
  $\sigma_1,\dots, \sigma_k$ can be chosen w.l.o.g. to use pairwise
  disjoint color sets and thus, easily extend to a proper coloring $\sigma$
  of $G$.  Due to the hierarchical structure of strong modules, every
  strong module of $G$ must be contained in one of the $M_i$,
  $1\leq i\leq k$.  Since, for every $G[M_i]$, the coloring $\sigma$
  restricted to $G[M_i]$ is a strictly hierarchical coloring and since
  Def.\ \ref{def:genhierstrict} imposes no further conditions on prime and
  series modules, the graph $G$ has a strictly hierarchical coloring.
 
  Otherwise, if $V$ is a parallel module, then every $\sigma_i$ can be
  chosen, w.l.o.g., to use only color sets $\{1,\dots,s_i\}$ with
  $s_i=|\sigma_i(M_i)|$.  Since $V$ is a parallel module, there are no
  edges between distinct $M_i$ and $M_j$. Hence, the colorings
  $\sigma_1,\dots, \sigma_k$ easily extend to a proper coloring $\sigma$ of
  $G$.  Furthermore, it obviously satisfies
  $\sigma(M_i)\subseteq \sigma(M_j)$ if and only if $s_i\le s_j$, and thus
  $\sigma(M_i)\cap\sigma(M_j)\in\{\sigma(M_i),\sigma(M_j)\}$.  This and the
  fact that $\sigma$ restricted to each $G[M_i]$ is a strictly hierarchical
  coloring, implies that $\sigma$ is a strictly hierarchical coloring of
  $G$.
\end{proof}

We next show that the (strictly) hierarchical colorings are a direct
generalization of the hierarchical colorings of cographs.
\begin{lemma}
  Let $G$ be a cograph and $\sigma$ a coloring of $G$. Then $\sigma$ is
  hierarchical if and only if it is hierarchical w.r.t.\ some binary cotree
  $(T,t)$ of $G$. Moreover, $\sigma$ is strictly hierarchical if and only
  if it is hierarchical w.r.t.\ all cotrees of $G$.
  \label{lem:wotree}
\end{lemma}
\begin{proof}
  Let $G$ be a cograph with modular decomposition tree $(\T,\t)$.  We will
  consider inner nodes $u$ in $\T$ whose children $u_i$ are given by
  $G(u_i)=G[M_i]$ for $\{M_1,\dots,M_k\}=\Pmax(G(u))$. We then replace
  every node $u$ of $(\T,\t)$ and its children $u_1,\dots,u_k$ by
  (specified) binary trees $(T_u,t_u)$ with root $u$ and leaves
  $u_1,\dots,u_k$. If $u$ is series (resp.\ parallel) we put $t_{u}(v)=1$
  (resp.\ $t_{u}(v)=0$) for all inner vertices $v$ in $T_u$.  It is an easy
  exercise to verify that the resulting tree $(T,t)$ is indeed a cotree of
  $G$.

  Suppose that $\sigma$ is hierarchical.  If $u$ is a series node, the
  color sets are disjoint and we have $G(u)=\join_{i=1}^k G(u_i)$. Hence,
  for any binary tree $(T_u,t_u)$ its inner nodes represent joins and
  clearly satisfy \AX{(K2)}.  Next, consider a parallel node $u$ in
  $\T$. Hence, there is at least one child of $u$, say $u_1$, such that
  $\sigma(M_i)\subseteq \sigma(M_1)$ for all $1\le i \le k$. Thus we can
  use a caterpillar corresponding to the ``pairwise'' cograph structure
  $((\dots((G[M_1]\cupdot G[M_{i_2}])\cupdot G[M_{i_3}])\dots)\cupdot
  G[M_{i_k}])$ with $\{i_2,\dots i_k\} = \{2,\dots,k\}$ arbitrarily
  chosen. In the resulting tree, the root $u$ and all its newly constructed
  inner nodes are labeled $0$. Clearly this satisfies \AX{(K3)} in every
  step.  Taken together, these constructions turn $(\T,\t)$ into a binary
  cotree $(T,t)$ that such $G$ becomes an hc-cograph w.r.t.\ $(T,t)$, and
  thus $\sigma$ is hierarchical w.r.t.\ $(T,t)$.

  Suppose now that $\sigma$ is strictly hierarchical.  If $u$ is a series
  node, then we can use exactly the same arguments as above to conclude
  that, after replacing all series nodes $u$ by arbitrary binary trees
  $(T_u,t_u)$, the resulting tree satisfies \AX{(K2)}.

  Suppose now that $u$ is a parallel node and let $(T_u,t_u)$ be an
  arbitrary binary tree.  We first show that for every inner vertex $w$ of
  $T_u$ we have $\sigma(V(G(w)))=\sigma(M_j)$ for some
  $j\in \{1,\dots,k\}$.  So let $w$ be an  arbitrary vertex. Let
  $I\subseteq \{1,\dots,k\}$ be a maximal subset such that $w$ is an
  ancestor of $u_i$ for all $i\in I$.  Since $\sigma$ is strictly
  hierarchical and thus hierarchical, there is a $j\in I$ such that
  $\sigma(M_i)\subseteq \sigma(M_j)$ for all $i\in I$.  Note, $G(w) = G[M]$
  with $M=\cupdot_{i\in I}M_i$.  Taken the latter two arguments together,
  we can conclude that $\sigma(G(w)) = \sigma(M) = \sigma(M_j)$.  The
  latter is, in particular, also true for the children $w_1$ and $w_2$ of
  $w$, i.e., $\sigma(G(w_1))=\sigma(M_{i_1})$ and
  $\sigma(G(w_2))=\sigma(M_{i_2})$ for some $i_1,i_2\in I$.  This and the
  fact that $\sigma$ is strictly hierarchical implies that
  $\sigma(G(w_1))\cap \sigma(G(w_2))\in \{\sigma(G(w_1)),\sigma(G(w_2))\}$
  Note, $w$ corresponds to $G(w) = G(w_1)\cupdot G(w_2)$.  Taken the latter
  two arguments together, implies that \AX{(K3)} holds for every inner
  vertex of $(T_u,t_u)$.  

  In summary, we can therefore replace $(\T,\t)$ by an arbitrary binary
  tree $(T,t)$ that displays $(\T,\t)$. Thus $\sigma$ is a hierarchical
  coloring of $G$ w.r.t.\ every cotree $(T,t)$ of $G$.

  Now consider the reverse implications.  The modular decomposition tree
  $(\T,\t)$ is obtained from any cotree $(T,t)$ of $G$ by stepwise
  contraction of all non-discriminating edges $uv$ i.e., those with
  $t(u)=t(v)$, into a new vertex $w_{uv}$ with label $\t(w_{uv})=t(u)$
  (cf.\ \cite{Boecker:98}).  The vertices of $\T$ are exactly the strong
  modules of $G$.  Again, we consider inner nodes $u$ in $\T$ whose
  children $u_i$ are given by $G(u_i)=G[M_i]$ for
  $\{M_1,\dots,M_k\}=\Pmax(G(u))$.  By construction, every inner vertex $v$
  on the path between $u_i$ and $u$ in $T$ is labeled $t(v)=t(u)$.  Since
  the definition of (strictly) hierarchical coloring does not impose
  constraints on series nodes $u$ (in which case $G(u)$ is connected),
  there is nothing to show for the case $t(u)=1$.  Hence, suppose that
  $t(u)=0$.

  Suppose now that $\sigma$ is a hierarchical coloring w.r.t.\ some binary
  cotree $(T,t)$. Let $v$ be such an inner vertex on the path between $u_i$
  and $u$. Clearly, $\sigma(G(v))$ is the union of the color set of some of
  the modules $M_1,\dots,M_k$.  This and the fact that \AX{(K3)} is
  satisfied for every vertex in $T$ implies $\sigma(G(v))=\sigma(G(u_j))$
  for some successor $u_j$ of $v$. Since the latter is, in particular, true
  for $u$, we have $\sigma(G(u))=\sigma(G(u_j)) = \sigma(M_j)$ for some
  $j\in \{1,\dots,k\}$. This and $\sigma(M_i)\subseteq \sigma(G(u))$
  implies that $\sigma$ is hierarchical.

  Finally, suppose $\sigma$ is a hierarchical coloring w.r.t.\ every binary
  cotree $(T,t)$, that is the disjoint union of the $G(u_i)$ to obtain
  $G(u)$ is constructed in an arbitrary order. In particular, therefore,
  for every pair $u_i\ne u_j$ of children of $u$ in $\T$ there is a binary
  cotree in which $u_i$ and $u_j$ have a common parent. By \AX{(K3)},
  therefore, we have $\sigma(M_i)\subseteq \sigma(M_j)$ or
  $\sigma(M_j)\subseteq \sigma(M_i)$ for all $1\le i< j\le k$, and thus
  $\sigma$ is strictly hierarchical.
\end{proof}

An an immediate consequence, we can restate Thm.~\ref{thm:greedy} in the
following form:
\begin{corollary}
  Let $G$ be a cograph. Then, $\sigma$ is a greedy coloring if and only if
  it is strictly hierarchical.
\end{corollary}

Definitions \ref{def:genhier} and \ref{def:genhierstrict} impose no
conditions on the prime nodes in the MD tree. Motivated by the equivalence
of $(T,t)$-minimal and hierarchical colorings of cographs, it seems natural
to consider colorings in which all strong modules are minimally colored.

\begin{definition}
  A coloring $\sigma$ of $G$ is \emph{modularly-minimal} if it satisfies
  $|\sigma(G[X])|=\chi(G[X])$ for all $X\in\mathcal{M}(G)$.
\end{definition}
Clearly, not every $\chi(G)$-coloring is also modularly-minimal.  As an
example, consider the disconnected graph $G = K_3 \cupdot P_3$ whose
components (and thus strong modules) are isomorphic to a $K_3$ and an
induced path $P_3$ on three vertices. Clearly, $\chi(G)=3$. A 3-coloring
$\sigma$ of $G$ that uses all three colors for the $P_3$ is still minimal,
but not modularly-minimal since $\chi(P_3)=2$.

\begin{lemma}
  Every modularly-minimal coloring of a graph is hierarchical.
  \label{lem:modmin-hc}
\end{lemma}
\begin{proof}
  For every graph $G$ with proper coloring $\sigma$ and connected
  components $G_i$ holds $\chi(G)=\max_i \chi(G_i)$. Hence, there is a
  connected component $G_j$ with $\chi(G)=\chi(G_j)$ and thus
  $\sigma(G)=\sigma(G_j)$.  Now assume that $\sigma$ is a modularly-minimal
  coloring of $G$.  Then for every parallel module $X$ with
  $\Pmax(G[X])=\{M_1,\dots,M_k\}$ holds
  $\chi(G[X])=\chi(G[M_j])=|\sigma(M_j)|$ and thus $\sigma(X)=\sigma(M_j)$
  for some $j\in\{1,\dots,k\}$. Therefore,
  $\sigma(M_i)\subseteq \sigma(M_j)$ for all $i\in\{1,\dots,k\}$, and thus,
  $\sigma$ is hierarchical.
\end{proof}

However, not every modularly-minimal coloring is strictly hierarchical.  As
an example, consider the cograph $G=K_4\cupdot K_2\cupdot K_2$. Its strong
modules are the singletons $\{v\}$, $v\in V$, the set $V$ and the vertex
sets of its connected components. We have $\chi(G)=\chi(K_4)=4$ and
$\chi(K_2)=2$. Consider a coloring $\sigma$ in which the two copies of
$K_2$ are colored $\{1,2\}$ and $\{3,4\}$ respectively. Clearly $\sigma$ is
modularly-minimal and hierarchical w.r.t.\ the binary cotree
$(K_4\cupdot K_2)\cupdot K_2$ but not w.r.t.\ the alternative binary cotree
$(K_2\cupdot K_2)\cupdot K_4$. In the latter case, (K3) is violated.

\begin{theorem}
  \label{thm:hc-modmin}
  Let $G$ be a cograph and $\sigma$ be a coloring of $G$.  The following
  statements are equivalent:
  \begin{enumerate}
  \item $\sigma$ is a hierarchical coloring w.r.t. some binary cotree
    $(T,t)$
  \item $\sigma$ is $(T,t)$-minimal
  \item $\sigma$ is modularly-minimal.
  \item $\sigma$ is hierarchical. 
  \end{enumerate}
\end{theorem}
\begin{proof}
  The equivalence (1) and (2) is provided by Theorem \ref{thm:mrc}, the
  equivalence (1) and (4) is provided by Lemma \ref{lem:wotree}.  We show
  first that (1) implies (3).  Suppose that $\sigma$ is a hierarchical
  coloring w.r.t.\ the binary cotree $(T,t)$ of $G$.  By Lemma
  \ref{lem:allStrongInCotree}, all strong modules of $G$ belong to some
  vertex in $T$. By Thm.\ \ref{thm:mrc}, $\sigma$ is $(T,t)$-minimal.  The
  latter two arguments imply that $|\sigma(G[M])|=\chi(G[M])$ for all strong
  modules of $G$.  Hence, $\sigma$ is modularly-minimal.  Finally, (3)
  implies (4) by Lemma \ref{lem:modmin-hc}, which completes the proof.
\end{proof}

The tight connection between hierarchical and modularly-minimal colorings
for graphs prompts the question whether minimal colorings like (strictly)
hierarchical colorings can be constructed for all graphs.
\begin{theorem}\label{thm:allG-modmin}
  Every graph admits a modularly-minimal coloring.
\end{theorem}
\begin{proof}
  We proceed by induction on the number of vertices. Obviously, $K_1$ has a
  modularly-minimal coloring.  Let $G=(V,E)$ be a graph and suppose that
  every graph with less than $|V|$ vertices has a modularly-minimal
  coloring. Let $\Pmax(G) = \{M_1,\dots,M_k\}$ be the (unique) maximal
  modular partition of $G$. Let $\sigma$ be a $\chi(G)$-coloring of $G$.

  By induction hypothesis, every graph $G[M_i]$ has a modularly-minimal
  coloring. Since $\chi(G[M_i])\leq |\sigma(G[M_i])|$, we can reuse the
  colors in every $\sigma(G[M_i])$ to obtain a $\chi(G[M_i)$-coloring of
  $G[M_i]$ that is modularly-minimal for each $G[M_i]$.  This results in a
  new coloring $\sigma^*$ of $G$.
 
  We show that $\sigma^*$ is a proper coloring of $G$.  Let $uv$ be an edge
  of $G$.  If $u,v$ reside in the same module $M_i\in \Pmax(G)$ we have
  $\sigma^*(u)\neq \sigma^*(v)$ since every module of $\Pmax(G)$ is
  modularly-minimal colored.  If $u,v$ reside in the distinct modules
  $M_i, M_j\in \Pmax(G)$ then all vertices in $M_i$ are adjacent to all
  vertices $M_j$, which implies
  $\sigma(G[M_i])\cap \sigma([G[M_i]]) = \emptyset$.  Since the recoloring
  $\sigma^*$ uses only those colors in each modules that already appeared,
  we have $\sigma^*(G[M_i])\cap \sigma^*([G[M_i]]) = \emptyset$.  Hence,
  $\sigma^*(u)\neq \sigma^*(v)$, i.e., $\sigma^*$ is proper.  In
  particular, we have not introduced new colors and thus, $\sigma^*$
  remains a $\chi(G)$-coloring.
      
  Because of the hierarchical structure of strong modules, every
  strong module $X$ of $G$ distinct from $V$ is contained in $M_i$ for some
  $M_i\in\Pmax(G)$ and thus satisfies, by construction,
  $|\sigma^*(G[X])|=\chi(G[X])$. For the strong module $V$ we have
  $|\sigma^*(G[V])| = \chi(G)$. Thus $\sigma^*$ is a modularly-minimal
  coloring of $G$.
\end{proof}

\begin{corollary}
  Every graph admits a modularly-minimal strictly hierarchical coloring.
\end{corollary}
\begin{proof}
  By Thm.\ \ref{thm:allG-modmin}, every graph $G$ admits a modularly-minimal
  coloring $\sigma$, that is, by Lemma \ref{lem:modmin-hc}, a hierarchical
  coloring of $G$. For every parallel module $X$ we define an arbitrary
  order on the color set $\sigma(X)$. As outlined in the proof of Lemma
  \ref{lem:modmin-hc} every module $M_i\in\Pmax(G[X])$ is colored with
  $\chi(G[M_i])$ colors of $\sigma(M_i)\subseteq\sigma(X)=\sigma(M_j)$ for
  some module $M_j\in \Pmax(G[X])$. Similar to the proof of
  Thm.~\ref{thm:strhier} we recolor each module $M_i$ with colors
  $1,\dots,\chi(M_i)$.  The resulting coloring $\sigma'$ is still
  modularly-minimal and satisfies $\sigma(M_i)\subseteq\sigma(M_j)$
  whenever $\chi(G[M_i])\le \chi(G[M_j])$. It follows that $\sigma'$ is
  strictly hierarchical.
\end{proof}

Modularly-minimal colorings provide a useful device to design efficient
coloring algorithms for certain hereditary graph classes. More precisely, a
polynomial-time coloring algorithm can be devised for every hereditary
graph class for which a minimal coloring can be constructed efficiently
given minimal colorings of its strong modules.  Such an algorithm is
outlined in Alg.\ \ref{alg:modmin} and used to show that so-called
$P_4$-sparse graphs can be modularly-minimal colored in polynomial time.

\begin{algorithm}[t]
  \caption{Modularly-minimal coloring a graph $G$ with MD tree $(\T,\t)$.}
  \label{alg:modmin}
  \algsetup{linenodelimiter=}
  \begin{algorithmic}[1]
    \REQUIRE Graph $G$ and MD tree $(\T,\t)$
    \STATE Initialize a coloring $\sigma$ s.t.\ all $v \in V(G)$
           have different colors
    \FORALL[from bottom to top where each $u$ is processed
               after all its children have been processed]{$u\in V^{0}(T)$}
       \IF {$u$ is parallel} 
          \STATE $\mathcal{G} \leftarrow \{G(w)\colon w\in\child(u)\}$ 
          \STATE $G^* \leftarrow \argmax_{w\in\child(v)} |\chi(G(w))|$
          \STATE $S \leftarrow \sigma(V(G^*))$ 
          \FOR {$H\in\mathcal{G}\setminus \{G^*\}$} 
             \STATE randomly choose an injective map $\phi:\sigma(H)\to S$
             \FORALL {$x\in H$}
                \STATE $\sigma(x)\leftarrow \phi(\sigma(x))$  
             \ENDFOR
          \ENDFOR
       \ELSIF{$u$ is \emph{prime}} 
          \STATE Construct a modularly-minimal coloring of $G(u)$
              with colors contained in $\sigma(G(u))$
              and adjust $\sigma$ accordingly \label{alg:prime-color}
       \ENDIF
    \ENDFOR
  \end{algorithmic}
\end{algorithm}

\begin{lemma}\label{lem:alg:modmin}
  Given a graph $G$ and corresponding MD tree $(\T,\t)$, Alg.\
  \ref{alg:modmin} returns a modularly-minimal coloring $\sigma$.
\end{lemma}
\begin{proof}
  We need to show that the Algorithm constructs a modularly-minimal
  coloring for every vertex $u$ of $T$. For the leaves this is trivial.  By
  analogous arguments as in the proof of Thm.\ \ref{Thm:Algo1}, we can
  conclude that the coloring of $G(u)$ is modularly-minimal provided
  that $u$ is a series or parallel node.  Moreover, if $G(u)$ is prime,
  then the entire subgraph $G(u)$ is modularly-minimally colored 
  with colors used for the subgraphs $G(u_i)$, $u_i\in \child(u)$.  
	A modularly-minimal  coloring exists by Thm.\ \ref{thm:allG-modmin}. In particular, since the
  color sets of $\sigma(G(u_i))$ and $\sigma(G(u_j))$ of any two distinct
  children $u_i,u_j\in \child(u)$ are disjoint while $u$ is processed,
  $\sigma(G(u))$ contains enough colors to obtain a $\chi(G(u))$ coloring,
  which completes the proof.
\end{proof}
It has to be noted, of course, that  constructing a modularly-minimal 
	coloring for prime nodes $u$ is a hard problem in
  general. However, \emph{if} it can be solved efficiently for each prime
  node, then Alg.\ \ref{alg:modmin} provides an efficient algorithm to
  construct a modularly-minimal coloring of $G$. Of course, this is
  trivially true for cographs since these lack prime nodes in the MD
  tree. The following example of $P_4$-sparse graphs shows that there are
  also interesting graph classes for which this is possible in a nontrivial
  manner.

A graph $G$ is $P_4$-sparse if each of its five-vertex induced
subgraphs contains at most one induced path on four vertices, a so-called $P_4$
\cite{Jamison:92}. They form a class of frequently studied
generalization of cographs. As shown in \cite{Jamison:92}, $G$ is
$P_4$-sparse if and only if every induced subgraph of $H$ of $G$ with at
least two vertices satisfies exactly one the following conditions: (i) $H$
is disconnected, (ii) $\overline{H}$ is disconnected, or (iii) $H$ is a
spider (defined below). In particular, therefore, every prime node in the
modular decomposition tree of $G$ is a spider, and its children except
$G[R]$ (unless $R=\emptyset$) are leaves corresponding to the vertices of
the body and the legs. In order to construct a modularly-minimal coloring
of a $P_4$-sparse graph, it therefore suffices to find a suitable way of
coloring spiders.

\begin{definition} \cite{Jamison:92,Nastos:12} A graph $G$ is a \emph{thin
    spider} if its vertex set can be partitioned into three sets $K$, $S$,
  and $R$ so that (i) $K$ is a clique; (ii) $S$ is a stable set; (iii)
  $|K|=|S|\ge 2$; (iv) every vertex in $R$ is adjacent to all vertices of
  $K$ and none of the vertices of $S$; and (v) each vertex in $K$ has a
  unique neighbor in $S$ and \emph{vice versa}.  \newline A graph $G$ is a
  \emph{thick spider} if its complement $\overline{G}$ is thin spider.
\end{definition}
The sets $K$, $S$, and $R$ are usually referred to as the \emph{body}, the
set of \emph{legs}, and \emph{head}, resp., of a thin spider. By definition
the head $R$ of a thin spider is a strong module in $G$, while all other
strong modules of a thin spider are trivial, that is, the vertices in
$K\cup S$ are all trivial modules (due to the 1-1 correspondence between
the vertices in $K$ and $S$ which precludes that any larger subset of $K$
or $S$ could be a module), see also \cite{Giakoumakis:97}. The same is true
for thick spiders since (strong) modules are preserved under
complementation.

It is not difficult to determine the chromatic number of a spider and to
construct a corresponding coloring:
\begin{lemma}
  If $G$ is a thin spider, $\chi(G)=\chi(R)+|K|$, if $G$ is thick spider,
  $\chi(G)=\chi(\overline{R})+|S|$, with $\chi(R)=0$ if the head $R$ is
  empty.
  \label{lem:spidercolor}
\end{lemma}
\begin{proof} 
  First, let $G$ be a thin spider with non-empty head $R$. Then
  $\chi(G)\ge |K|+\chi(R)$ since by definition every vertex of $K$ is
  connected with every vertex of $R$, i.e., the color sets of $R$ and $K$
  must be disjoint. Since the body $K$ is a clique, it requires
  $\chi(K)=|K|$ colors. Each leg $x\in S$ is connected to a unique vertex
  $x'\in K$ of the body. Since $|S|=|K|\ge 2$, one can always choose a
  color in $\sigma(K)$ different from $\sigma(x')$ to color $x$, hence
  $|K|+\chi(R)$ are sufficient to color $G$. If the head $R$ is empty, the
  $|K|$ colors of the body are sufficient by the same argument.

  Now suppose $G$ is a thick spider, i.e., $\overline{G}$ is a thin spider.
  Thinking of $K$, $S$, and $R$ as induced subgraphs of $G$, we note that
  $\overline{S}$ is a clique in $\overline{G}$, $\overline{K}$ is an
  independent set in $\overline{G}$, and every vertex of $\overline{R}$ is
  connected to every vertex of $\overline{S}$ and none of
  the vertices in $\overline{K}$. Thus
  $\chi(\overline{G})\ge |S|+\chi(\overline{R})$. Each vertex of
  $x\in\overline{K}$ is adjacent to all but one vertex in $\overline{S}$,
  which we call $x'$. Since there is a 1-1 correspondence between $x$ and
  $x'$, we can give them same color, i.e., $|S|+\chi(\overline{R})$ are
  sufficient. Again, if the head $R$ is empty, the $|S|$ colors of the legs
  are sufficient.
\end{proof}

\begin{corollary}
  A coloring of a $P_4$-sparse graph is modularly-minimal if each spider
  $(K,S,R)$ in its MD tree is colored such that $\sigma(S)=\sigma(K)$ and
  $\sigma$ is a modularly-minimal coloring of its head.
  \label{cor:p4s-col}
\end{corollary}
\begin{proof}
  By definition, every prime vertex $u$ in the MD tree of a $P_4$-sparse
  graphs must be a spider $G(u) = (K,S,R)$. Therefore, by Lemma
  \ref{lem:alg:modmin} and the construction in Alg.\ \ref{alg:modmin}, it
  suffices to show that a $P_4$-sparse graph is modularly-minimal colored,
  if each spider $G(u) = (K,S,R)$ in its MD tree is modularly-minimal
  colored.  Note, $\Pmax(G(u)) = \cup_{v\in K\cup S}\{\{v\}\} \cup \{R\}$,
  that is, each child of $u$ corresponds to a vertex $v\in K\cup S$ and
  $R$.  By construction all color sets of the children of $u$ are pairwise
  disjoint. Using that $|K|=|S|$, we can obtain $\sigma(S)=\sigma(K)$
    by reusing $\sigma(K)$ to color $S$ as described in the proof of Lemma
    \ref{lem:spidercolor}.  Furthermore, $\sigma$ is a modularly-minimal
  coloring of the head $R$. As outlined in the proof of Lemma
  \ref{lem:spidercolor}, this yields a minimal coloring of the spider
  $(K,S,R)$. Since $\Pmax(G(u)) = \cup_{v\in K\cup S}\{\{v\}\} \cup \{R\}$,
  every strong module in the spider $G(u)$ is \emph{even}
  modularly-minimally colored, which completes the proof.
\end{proof}

\begin{corollary}
  A modularly-minimal coloring of a $P_4$-sparse graph $G=(V,E)$ can be
  computed in $O(|V|+|E|)$ time.
\end{corollary}
\begin{proof}
  The modular decomposition of $G$ can be obtained in $O(|V|+|E|)$ time
  \cite{McConnell:99}. Replace Line \ref{alg:prime-color} in Alg.\
  \ref{alg:modmin} by the construction as in Cor.\ \ref{cor:p4s-col} and
  consider a vertex $u$ in the MD tree.  If $u$ is series, there is nothing
  to do. If $u$ is parallel, the recoloring of the connected components can
  be performed in $O(|V(G(u))|)$ time. If $u$ is prime, i.e., $G(u)$ is
  spider, we only need to recolor its legs with the colors of the body as
  in the proof of Lemma \ref{lem:spidercolor}, which can also be done in
  $O(|V(G(u))|)$ time. Each node $u$ of the MD tree corresponds to strong
  module $M = V(G(u))\in \mathcal{M}(G)$ and thus can be handled in
  $O(|M|)$ time. By \cite[Thm.22]{McConnell:05} we have
  $O(\sum_{M\in\mathcal{M}(G)} |M|)\le 2|E|+3|V|$, and thus the total
  effort is $O(|V|+|E|)$, that is, linear in the size of $G$.
\end{proof}

\section{Concluding Remarks}

The existence of modularly-minimal colorings can serve as guiding principle
for recursive algorithms to compute the chromatic number. In particular,
whenever it is possible to efficiently compute the chromatic number of a
prime module $X$ from the quotient graph $G[X]/\Pmax(G[X])$ and chromatic
number of the modules $M\in\Pmax(G[X])$, one obtains an efficient algorithm
for this purpose. The virtue of Thm.\ \ref{thm:allG-modmin} in this context
is to relieve the need for controlling the colorings of the
modules. Examples of such a construction are the recursive computation of
the chromatic number for ($P_5$,gem)-free graphs \cite{Bodlaender:05} and
$P_4$-tidy graphs \cite{Giakoumakis:97b}. Here, we provided a further
algorithm to compute a minimal coloring of $P_4$-sparse graphs in
polynomial time. The approach may be useful more generally for vertex
colorings of graphs with forbidden subgraphs surveyed in
\cite{RS:04}. It is tempting to consider modularly-minimality as a
guiding principle for other optimization problems.

Moreover, the general idea of (strictly) hierarchical or modularly-minimal
coloring is not restricted to the modular decomposition. Many interesting
classes of graphs admit recursive constructions
\cite{Proskurowski:81,Noy:04}. For every graph class of this type,
one can ask whether minimal colorings can be constructed from minimal
colorings of the constituents, i.e., whether recursively minimal
colorings exist.

\subsection*{Acknowledgments} 

This work was support in part by the German Research Foundation (DFG, STA
850/49-1), the German Federal Ministry of Education and Research (BMBF,
project no.\ 031A538A, de.NBI-RBC), and the Mexican Consejo Nacional de
Ciencia y Tecnolog{\'i}a (CONACyT, 278966 FONCICYT2; scholarship CVU 901154).
\bibliographystyle{plain}
\bibliography{cographcoloring}
\end{document}